\documentclass{article}

\usepackage{amssymb,amsmath,amsthm}
\usepackage{graphicx}
\usepackage{url}
\usepackage{hyperref}
\usepackage{color}

\newtheorem{theorem}{Theorem}
\newtheorem{definition}{Definition}

\newtheorem{lemma}[theorem]{Lemma}

\newcommand{\junk}[1]{}

\newcommand{\R}{{\mathbb{R}}}

\newcommand{\N}{{\mathbb{N}}}
\newcommand{\Z}{{\mathbb{Z}}}

\newcommand{\E}{\mathbb{E}}
\renewcommand{\Pr}{\mathbb{P}}

\newcommand\eps{\varepsilon}

\newcommand\FF{\mathcal{F}}
\newcommand\HH{\mathcal{H}}

\newcommand\TT{\mathcal{T}}

\renewcommand\AA{\mathcal{A}}

\newcommand{\eqdef}{:=}
\newcommand{\maps}{\colon}

\DeclareMathOperator{\herdisc}{herdisc}
\DeclareMathOperator{\disc}{disc}

\DeclareMathOperator{\rank}{rank}
\DeclareMathOperator{\tr}{tr}

\begin{document}
\title{Tighter Bounds for the Discrepancy of Boxes and Polytopes}
\author{Aleksandar Nikolov\\
  Department of Computer Science\\
  University of Toronto\\
  \url{anikolov@cs.toronto.edu}}
\date{}

\maketitle

\begin{abstract}
  Combinatorial discrepancy is a complexity measure of a collection of sets which quantifies how well the sets in the collection can be simultaneously balanced. More precisely, we are given an n-point set $P$, and a collection $\FF = \{F_1, ..., F_m\}$ of subsets of $P$, and our goal is color $P$ with two colors, red and blue, so that the maximum over the $F_i$ of the absolute difference between the number of red elements and the number of blue elements (the discrepancy) is minimized. Combinatorial discrepancy has many applications in mathematics and computer science, including constructions of uniformly distributed point sets, and lower bounds for data structures and private data analysis algorithms.  

We investigate the combinatorial discrepancy of geometrically defined systems, in which $P$ is an n-point set in $d$-dimensional space ,and $\FF$ is the collection of subsets of $P$ induced by dilations and translations of a fixed convex polytope $B$. Such set systems include systems of sets induced by axis-aligned boxes, whose discrepancy is the subject of the well known Tusn\'ady problem. We prove new discrepancy upper and lower bounds for such set systems by extending the approach based on factorization norms previously used by the author and Matou\v{s}ek. We improve the best known upper bound for the Tusn\'ady problem by a logarithmic factor, using a result of Banaszczyk on signed series of vectors. We extend this improvement to any arbitrary convex polytope $B$ by using a decomposition due to Matou\v{s}ek. Using Fourier analytic techniques, we also prove a nearly matching discrepancy lower bound for sets induced by any fixed bounded polytope $B$ satisfying a certain technical condition.

We also outline applications of our results to geometric discrepancy, data structure lower bounds, and differential privacy.
\end{abstract}

\section{Introduction} 

As usual, we define the combinatorial discrepancy of a system $\FF$ of
subsets of some finite set $P$ as:
\[
\disc\FF \eqdef \min_{\chi\maps P \to \{-1, 1\}}\max_{F \in \FF}{ \sum_{p \in
    F}{\chi(p)}  }.
\]
The book of Matou\v{s}ek~\cite{Matousek-book} provides more background
on combinatorial and geometric discrepancy. For a reference to the
large number of applications of combinatorial discrepancy to geometric
discrepancy, numerical methods, and computer science, see the book of
Chazelle~\cite{Chazelle-book}. A recent application to private data
analysis is described in the author's PhD thesis~\cite{thesis}. We
give more details on applications to geometric discrepancy, numerical
integration, data structures, and differential privacy in
Section~\ref{sect:apps}.

Let $\AA_d$ be the family of \emph{anchored} axis-aligned boxes in
$\R^d$: $\AA_d \eqdef \{A(x): x \in \R^d\}$, where $A(x) \eqdef \{y
\in \R^d: y_i \le x_i\ \forall i \in [d]\}$ and $[d] \eqdef \{1,
\ldots, d\}$. This is a slight abuse of terminology: $A(x)$ is not a
box, but rather a polyhedral cone. Usually $A(x)$ is defined to be
anchored at $0$, i.e.~it is defined as the set $\{y \in \R^d: 0 \le
y_i \le x_i\ \forall i \in [d]\}$. However, we prefer to ``anchor''
$A(x)$ at infinity. This convention does not affect any of the results
in the paper, and allows us to avoid some minor technicalities.

For an $n$-point set $P \subset \R^d$, we denote by $\AA_d(P) \eqdef
\{A(x) \cap P: x \in \R^d\}$ the set system induced by anchored boxes
on $P$. (Note that this set system is finite, and, in fact, can have
at most $n^d$ sets.)  Tusn\'ady's problem asks for tight bounds on the
largest possible combinatorial discrepancy of $\AA_d(P)$ over all
$n$-point sets $P$, i.e.~for the order of growth of the function
$\disc(n,\AA_d) \eqdef \sup\{\disc \AA_d(P): P \subset \R^d, |P| =
n\}$ with $n$.  The best known upper bound for the Tusn\'ady problem
is $\disc(n, \AA_d) = O_d(\log^d n)$, and was recently proved by
Bansal and Garg~\cite{BansalG16}.  Their result improved on the prior
work of Larsen~\cite{larsen14} (see also the proof
in~\cite{disc-gamma2}), who showed that $\disc(n, \AA_d) =
O_d(\log^{d+1/2} n)$. Here, and in the rest of this paper, we use the
notation $O_p(\cdot), \Omega_p(\cdot)$, $\Theta_p(\cdot)$ to denote
the fact that the implicit constant in the asymptotic notation depends
on the parameter $p$.

In our first result, we improve the upper bounds above:

\begin{theorem}\label{thm:tusnady-ub}
  For any $d \ge 2$, 
  \[
  \disc(n,\AA_d) = O_d(\log^{d-1/2} n).
  \]
\end{theorem}

It was shown in~\cite{e8-tusnady,disc-gamma2} that $\disc(n,\AA_d) =
\Omega_d(\log^{d-1} n)$, so the upper bound above is within a
$O_d(\sqrt{\log n})$ factor from the lower bound. This brings us
tantalizingly close to a full resolution of Tusn\'ady's problem. 

More generally, let $\FF$ be a collection of subsets of $\R^d$, and,
for an $n$-point set $P \subset \R^d$, let $\FF(P)$ be the set system
induced by $\FF$ on $P$, i.e.~$\FF(P) \eqdef \{F \cap P: F \in
\FF\}$. We are interested in how the worst-case combinatorial discrepancy
of such set systems grows with $n$. This is captured by the
function $\disc(n, \FF) \eqdef \sup\{\disc\FF(P): P \subset \R^d, |P|
= n\}$. 

Let $K\subseteq \R^d$, and let's define $\TT_K$ to be the family of
images of $K$ under translations and homothetic transformations:
$\TT_K \eqdef \{tK + x: t \in \R_{+}, x \in \R^d\}$, where $\R_{+}$ is
the set of positive reals. If we take $Q \eqdef [0, 1]^d$, then it is
well-known that $\disc(n, \TT_Q) \le 2^d\disc(n, \AA_d)$ and,
therefore, Theorem~\ref{thm:tusnady-ub} implies $\disc(n,
\TT_Q) = O_d(\log^{d-1/2} n)$. Our next result shows that this bound
holds for any polytope $B$ in $\R^d$.

\begin{theorem}\label{thm:polytopes-ub}
  For any $d \ge 2$, and any closed convex polytope $B \subset \R^d$,
  \[
  \disc(n, \TT_B) = O_{d,B}(\log^{d-1/2} n).
  \]
\end{theorem}
With the same proof we can
establish the stronger fact that $\disc(n, \mathrm{POL}(\HH)) =
O_{d, \HH}(\log^{d-1/2}n)$, where  $\HH$ is a family of hyperplanes in
$\R^d$, and $\mathrm{POL}(\HH)$ is the set of all polytopes which can
be written as $\bigcap_{i = 1}^m{H_i}$, with each $H_i$ a
halfspace whose boundary is parallel to some $h \in \HH$. The best
previously known upper bound in this setting is also due to Bansal and
Garg~\cite{BansalG16}, and is equal to $O_{d,\HH}(\log^d n)$. 

Our final result is a nearly matching lower bound for any
generic convex polytope $B$. (See
Definition~\ref{defn:non-degenerate} for the meaning of
``generic''.)
\begin{theorem}\label{thm:lb-main}
  Let $B\subseteq \R^d$ be a generic convex polytope with
  non-empty interior. Then $\disc(n,\TT_B) =
  \Omega_{d,B}(\log^{d-1}n)$.
\end{theorem}
The best previously known lower bound on $\disc(n,\TT_B)$ was
$\Omega_{d,B}(\log^{(d-1)/2}n)$, which follows from a result of
Drmota~\cite{Drmota96} in geometric discrepancy theory. This lower
bound holds for any convex polytope $B$, but is quadratically weaker
than our lower bound.

We conjecture that the genericity condition in
Theorem~\ref{thm:lb-main} is not necessary, and, furthermore, that the
asymptotic constant in the lower bound need not depend on $B$. These
conjectures would be implied by certain estimates on the Fourier
spectrum of convex polytopes: see Section~\ref{sect:lb} for more
details. By contrast, the asymptotic constant in the upper bound in
Theorem~\ref{thm:polytopes-ub} has to depend on $B$, because we can
approximate the unit Euclidean ball $D^d$ in $\R^d$ arbitrarily well
with a convex polytope, and $\disc(n, \TT_{D^d}) = \Omega_d(n^{1/2 -
  1/(2d)})$~\cite{Alexander91}.

Together, these results give nearly tight estimates on the discrepancy
function $\disc(n, \TT_B)$ for almost any convex polytope $B$. Perhaps
surprisingly, they imply that the order of growth of $\disc(n, \TT_B)$
with $n$ is essentially independent of the particular structure of $B$, at least
when $B$ is generic, and, moreover, that order of growth is
achieved by the simplest possible example, a cube. In the next section
we describe applications of our results to geometric discrepancy,
quasi-Monte Carlo methods, data structure lower bounds, and
differential privacy.

\section{Applications}\label{sect:apps}

In this section we outline several applications of our discrepancy
upper and lower bounds. 

\subsection{Geometric Discrepancy and Numerical Integration}

Geometric discrepancy measures the irregularity of a distribution of
$n$ points in $[0,1]^d$ with respect to a family of distinguishing
sets. In particular, for an $n$-point set $P \subseteq [0,1]^d$ and a
family of measurable subsets $\FF$ of $\R^d$, we define the
discrepancy 
\[
D(P,\FF) \eqdef \sup_{F \in \FF}\left| |P \cap F| - \lambda_d(F \cap [0,1]^d)\right|,
\]
where $\lambda_d$ is the Lebesgue measure on $\R^d$. The smallest
achievable discrepancy over $n$-point sets with respect to $\FF$ is
denoted $D(n, \FF) \eqdef \inf\{D(P, \FF): P \subset \R^d, |P| = n\}$.
A famous result of Schmidt~\cite{Schmidt-VII} shows that $D(n, \AA_2)
= \Theta(\log n)$. The picture is much less clear in higher
dimensions. In a seminal paper~\cite{Roth54}, Roth showed that $D(n,
\AA_d) = \Omega_d(\log^{(d-1)/2} n)$ for any $d \ge 2$; the best known
lower bound in $d\ge 3$ is due to Bilyk, Lacey, and
Vagharshakyan~\cite{BilykLV08} and is $D(n, \AA_d)
=\Omega_d(\log^{(d-1)/2 + \eta_d} n)$, where $\eta_d$ is a positive
constant depending on $d$ and going to $0$ as $d$ goes to infinity. On
the other hand, the best known upper bound is $D(n, \AA_d) =
O_d(\log^{d-1}n)$ and can be achieved in many different ways, one of
the simplest being the Halton-Hammersley construction~\cite{Halton60,
  Hammersley60}. The book by Beck and Chen~\cite{BeckChen} calls the
problem of closing this significant gap ``the Great Open Problem'' (in
geometric discrepancy theory). See the book of
Matou\v{s}ek~\cite{Matousek-book} for further background on geometric
discrepancy.

There is a known connection between combinatorial and geometric
discrepancy. Roughly speaking, combinatorial discrepancy is an upper
bound on geometric discrepancy. More precisely, we have the following
transference lemma, which goes back to the work of Beck on Tusn\'ady's
problem~\cite{Beck81} (see~\cite{Matousek-book} for a proof). 

\begin{lemma}\label{lm:transference}
  Let $\FF$ be a family of measurable sets in $\R^d$ such that there
  is some $F \in \FF$ which contains $[0,1]^d$. Assume that
  $\frac{D(n,\FF)}{n}$ goes to $0$ as $n$ goes to infinity, and that
  $\disc(n,\FF) \le f(n)$ for a function $f$ that satisfies $f(2n) \le
  (2-\delta)f(n)$ for all $n$ and some fixed $\delta > 0$. Then there
  exists a constant $C_\delta$ that only depends on $\delta$, for
  which $D(n,\FF) \le C_\delta f(n)$. 
\end{lemma}

Lemma~\ref{lm:transference} and Theorem~\ref{thm:polytopes-ub} imply
that $D(n,\TT_B) = O_{d,B}(\log^{d-1/2}n)$ for any convex polytope $B$
in $\R^d$. This bound gets within an $O_{d,B}(\sqrt{\log n})$ factor
from the best bound known for axis-aligned boxes in $d$
dimensions. The tightest bound known prior to our work was
$O_{d,B}(\log^d n)$ and was also implied by the best previously known
upper bound on combinatorial discrepancy. 

Geometric discrepancy can be defined with any Borel probability
measure $\nu$ on $[0,1]^d$ in place of the Lebesgue measure: let's
call the resulting quantities $D(P, \FF, \nu)$ and $D(n, \FF, \nu)$. It
turns out that Lemma~\ref{lm:transference} holds with $D(n, \FF)$
replaced by $D(n, \FF, \nu)$, and, together with
Theorem~\ref{thm:tusnady-ub} we get that $D(n, \AA_d, \nu) =
O_d(\log^{d-1/2}n)$ for any Borel probability measure $\nu$ on $[0,1]^d$. This
bound has an application to the quasi-Monte Carlo method in numerical
integration. A version of the Koksma-Hlawka inequality for general
measures due to G\"otz~\cite{Gotz02} shows that for any real-valued
function $f$ on $[0,1]^d$ of bounded total variation in the sense of
Hardy and Krause, we have
\[
\left| \int_{[0,1]^d}f(x)d\nu(x) - \frac{1}{n} \sum_{x \in
    P}{f(x)}\right|
\le 
\frac{1}{n}V(f) D(P, \AA_d, \nu).
\]
Here, $V(f)$ is the Hardy-Krause variation of $f$. So, our upper bound
implies that for any function $f$ of constant total variation, and any
Borel measure $\nu$, we can numerically estimate the integral of $f$
with respect to $\nu$ by averaging the values of $f$ at $n$ points, and
the estimate has error $O_d(n^{-1} \log^{d-1/2}(n))$. Contrast this with the
$O_d(n^{-1/2})$ error rate achieved by the Monte Carlo method. The error
rate we achieve is within $O_d(\sqrt{\log n})$ of the best error rate
known for Lebesgue measure. Integration with respect to measures other
than the Lebesgue measure arises often, and a constructive
proof of our upper bound could have significant impact in practice. We
refer to the note~\cite{mcqmc} for an exposition of these
connections. 

\subsection{Range Searching Lower Bounds}\label{sect:ds}

In the dynamic range searching problem, we are given a \emph{range space}
$\FF(P)$, where $\FF$ is a collection of subsets of $\R^d$, and $P$ is
an $n$-point set in $\R^d$; our goal is to design a data structure
which keeps a set of weights $w \in G^P$ under updates, where
the weights come from a commutative group $G$. An update
specifies a point $p$ of $P$ and an element $g \in G$, and asks to
change the weight of $p$ to $w_p + g$. The data structure should be
able to answer range searching queries, where a query is specified by
a range $F \in \FF(P)$, and must return the answer $\sum_{p \in
  F}{w_p}$. The main question for this data structure problem is to
identify a tight trade-off between the update time and the query
time.

Fredman~\cite{Fredman82} first observed that many data structures for
the dynamic range searching problem can be identified with a matrix
factorization $A = UV$ of the incidence matrix $A$ of $\FF(P)$ into
two matrices $U$ and $V$ with integer entries. Following
Larsen~\cite{larsen14}, we define an oblivious data structure with
multiplicity $\Delta$ for the dynamic range searching problem for the
range space $\FF(P)$ as a factorization $A = UV$ of the incidence
matrix of $\FF(P)$ into matrices $U$ and $V$ with integer entries
bounded in absolute value by $\Delta$. The update time $t_u$ for such
a data structure equals the maximum number of non-zero entries in any
column of $V$; the query time $t_q$ equals the maximum number of
non-zero entries in any row of $U$. 

Our arguments imply the following result.
 \begin{theorem}\label{thm:ds}
   For any generic convex polytope $B$ in $\R^d$ there exists a
   family of point sets $P$ in $\R^d$ so that for any
   family of oblivious data structures with multiplicity $\Delta$ for
   the dynamic range counting problem with range space $\TT_B(P)$, we
   have $\sqrt{t_u t_q} =
   \Omega_{d,B}(\Delta^{-1}\log^{d}n)$. Conversely, for any convex
   polytope $B$ and any $n$-point set $P$ in $\R^d$ there exists an
   oblivious data structure with multiplicity $\Delta = 1$ and
   $\sqrt{t_ut_q} = O_{d,B}(\log^{d}n)$.
 \end{theorem}

 Theorem~\ref{thm:ds} implies that, as for the discrepancy question,
 the geometric mean of query and update time grows with $n$ at the
 same rate for any (generic) convex polytope $B$, and that order of
 growth is already achieved for orthogonal range
 searching. Remarkably, our upper and lower bounds are tight up to
 constants when the multiplicity is bounded. Our main contribution is
 the lower bounds, while the upper bounds follow from standard
 techniques.

\subsection{Differential Privacy}\label{sect:dp}

Range searching and range counting problems also naturally arise in
differential privacy. The setting here is that, given the range space
$\FF(P)$ as above, we have as input a database $D$ which is a multiset
of points from $P$. The goal is to output the number (with
multiplicity) of points in $D$ that fall in each range $F \in
\FF(P)$. However, the database $D$ could be sensitive because it may,
for example, encode the locations of different people. For this
reason, we want to approximate the range counts under the constraints
of differential privacy~\cite{DMNS,DworkR14}. Formally, a randomized
algorithm $\mathcal{M}$ (``mechanism'' in the terminology of
differential privacy) is $(\varepsilon, \delta)$-differentially
private if, for any two databases $D$, $D'$ that differ in the
location of single point, and all measurable subsets $S$ of the range
of $\mathcal{M}$, we have
\[
\Pr(\mathcal{M}(D) \in S) \le e^\varepsilon\Pr(\mathcal{M}(D) \in S) + \delta.
\]

We define the error of a mechanism $\mathcal{M}$ as the maximum of 
\[
\E\max_{F \in \FF(P)}|\mathcal{M}(D)_F - |D \cap F||
\]
over databases $D$, where the expectation is with respect to the
randomness of $\mathcal{M}$, and $\mathcal{M}(D)_F$ is the output that
$\mathcal{M}$ gives on input $D$ for the range $F$. Let
$\mathrm{opt}_{\varepsilon, \delta}(\FF(P))$ be the smallest
achievable error of an $(\varepsilon, \delta)$-differentially private
algorithm on $\FF(P)$, and let $\mathrm{opt}_{\varepsilon, \delta}(N,
\FF) = \sup \mathrm{opt}_{\varepsilon, \delta}(\FF(P))$, where the
supremum is over all $N$-point sets $P$ in $\R^d$.

Our techniques imply the following result.

\begin{theorem}\label{thm:dp}
  For any generic convex polytope $B$ in $\R^d$, for all small enough
  $\varepsilon$, and all $\delta$ small enough with respect to
  $\varepsilon$, we have
  \begin{align*}
    \mathrm{opt}_{\varepsilon, \delta}(N, \TT_B) &= 
    \Omega_{d,B}(\varepsilon^{-1}\log^{d-1}N),\\
    \mathrm{opt}_{\varepsilon, \delta}(N, \TT_B) &= 
    O_{d,B}(\varepsilon^{-1} \sqrt{\log 1/\delta}\log^{d+1/2}N).
  \end{align*}
  Moreover, the upper bound holds for any (not necessarily generic)
  convex polytope $B$.
\end{theorem}

Once again, our result shows that  the growth of the best possible
error under differential privacy with $N$ for the range counting
problem with ranges induced by a convex polytope $B$ does not depend
strongly on the structure of the particular polytope $B$.

\section{Preliminaries and Techniques}

In what follows $C_p$ and $c_p$ are constants that depend only on the
parameter $p$ and may change from one line to the next. All logarithms
are assumed to be in base $e$ (although usually this does not
matter). We use $\langle \cdot, \cdot \rangle$ for the standard
inner product on $\R^n$, and $\|\cdot \|_2$ for the corresponding
Euclidean norm. We use capital letters to denote matrices,
and lower case letters with indexes to denote matrix entries,
e.g.~$a_{ij}$ to denote the entry in row $i$ column $j$ of the matrix
$A$.  We also use $\mu_n$ for the standard Gaussian measure on
$\R^n$. (We avoid the more standard $\gamma_n$ to avoid confusion with
the $\gamma_2$ factorization norm.) We use $\sigma_{d-1}$ for the
uniform (rotation invariant, Haussdorff) probability measure on the
$d-1$ dimensional unit Euclidean sphere $S^{d-1}$ in $\R^d$. We use
$\theta_d$ for the Haar probability measure on the orthogonal group
$\mathbf{O}(d)$. 

\subsection{Hereditary Discrepancy}

Hereditary discrepancy is a robust version of combinatorial
discrepancy. For a set system $\FF$ of subsets of a set $P$, the
hereditary discrepancy of $\FF$ is defined by $\herdisc \FF = \max_{Q
  \subseteq P}{\disc \FF(Q)}$. Hereditary discrepancy is often more
tractable than discrepancy itself, both analytically, and
computationally. E.g.~while a non-trivial approximation to discrepancy
is in general $\mathsf{NP}$-hard~\cite{dischard}, hereditary
discrepancy can be approximated up to polylogarithmic
factors~\cite{apx-disc,disc-gamma2}. Importantly for us, this
robustness comes at no additional cost: for all collections of sets
$\FF$ that we study, it is easy to see that, for any $n$-point set $P$
in $\R^d$, $\disc(n, \FF) \ge \herdisc \FF(P)$. 

\subsection{The $\gamma_2$ factorization norm}

The $\gamma_2$ norm was introduced in functional analysis to study
operators that factor through Hilbert space. 
We say that an operator $u\maps X \to Y$ between Banach spaces $X$ and $Y$ factors through a Hilbert space
if there exists a Hilbert space $H$ and bounded operators $u_1\maps X \to
H$ and $u_2\maps H \to Y$ such that $u = u_2 u_1$. Then the $\gamma_2$
norm of $u$ is
\[
\gamma_2(u) \eqdef \inf \|u_1\| \|u_2\|,
\]
where the infimum is taken over all Hilbert spaces $H$, and all
operators $u_1$ and $u_2$ as above. Here $\|u_1\|$ and $\|u_2\|$ are
the operator norms of $u_1$ and $u_2$, respectively. The book of
Tomczak-Jaegermann~\cite{TJ-book} is an excellent reference on
factorization norms and their applications in Banach space theory.

In this work we will use the $\gamma_2$ norm of an $m\times n$ matrix
$A$, which is defined as the $\gamma_2$ norm of the linear operator
$u\maps\ell_1^n \to \ell_\infty^m$ with matrix $A$ (in the standard
bases of $\R^m$ and $\R^n$). In the language of matrices, this means
that 
\[
\gamma_2(A) \eqdef \inf\{ \|U\|_{2\to \infty}\|V\|_{1 \to 2}: A = UV\},
\]
where the infimum is taken over matrices $U$ and $V$, $\|V\|_{1\to 2}$
equals the largest $\ell_2$ norm of a column of $V$, and
$\|U\|_{2\to\infty}$ equals the largest $\ell_2$ norm of a row of
$U$. By a standard compactness argument, the infimum is achieved;
moreover, we can take $U \in \R^{m \times r}$ and $V \in \R^{r \times
  n}$, where $r$ is the rank of $A$. Yet another equivalent
formulation, which will be convenient for us, is that $\gamma_2(A)$ is
the smallest non-negative real $t$ for which there exist vectors $u_1,
\ldots, u_m, v_1, \ldots, v_n \in \R^r$ such that for any $i \in [m],
j \in [n]$, $a_{ij} = \langle u_i, v_j\rangle$ and $\|u_i\|_2 \le t$,
$\|v_j\|_2 \le 1$. For a proof of the not completely trivial fact that
$\gamma_2$ is a norm, see~\cite{TJ-book}. Given a matrix
$A$, $\gamma_2(A)$ can be computed efficiently by solving a
semidefinite program~\cite{LinialMSS07-signmatrices}.

Let us further overload the meaning of $\gamma_2$ by defining
$\gamma_2(\FF) = \gamma_2(A)$ for a set system $\FF$ with incidence
matrix $A$. We recall that the incidence matrix of a system $\FF =
\{F_1, \ldots, F_m\}$ of
subsets of a set $P = \{p_1, \ldots, p_n\}$ is defined as
\[
a_{ij} \eqdef
\begin{cases}
  1 &p_j \in F_i\\
  0 &p_j \not\in F_i\\
\end{cases}.
\]
In other worse, $\gamma_2(\FF)$ is the smallest non-negative real $t$
such that there exist vectors $u_1, \ldots, u_m, v_1, \ldots, v_n$
satisfying 
\[
\langle u_i, v_j\rangle =
\begin{cases}
  1 &p_j \in F_i\\
  0 &p_j \not\in F_i\\
\end{cases},
\]
and $\|u_i\|_2 \le t$, $\|v_j\|_2 \le 1$ for all $i \in [m], j \in
[n]$. 

In~\cite{apx-disc,disc-gamma2} it was shown that $\gamma_2(\FF)$ is,
up to logarithmic factors, equivalent to hereditary discrepancy:
\begin{align}
  \herdisc\FF &\le C\sqrt{\log m}\ \gamma_2(\FF)\label{eq:ub};\\
  \herdisc\FF &\ge c\frac{ \gamma_2(\FF)}{\log \rank A}. \label{eq:lb}
\end{align}
where $C,c >0$ are  absolute constants.

In~\cite{e8-tusnady,disc-gamma2} it was also shown that
$\gamma_2$ satisfies a number of nice properties which help in
estimating the norm of specific matrices or set systems. Here we only
need the following inequality, which holds for a set system $\FF =
\FF_1 \cup \ldots \cup \FF_k$, where $\FF_1, \ldots, \FF_k$ are set systems
over the same set $P$:
\begin{equation}
  \label{eq:union}
  \gamma_2(\FF) \le \sqrt{\sum_{i = 1}^k{\gamma_2(\FF_i)^2}}.
\end{equation}

We  also need the following simple lemma, which follows, e.g.~from the
results in~\cite{e8-tusnady,disc-gamma2}, but also appears in a
similar form in~\cite{larsen14}, and follows from standard dyadic
decomposition techniques.
\begin{lemma}\label{lm:factor}
  For any $d \ge 1$ there exists a constant $C_d$ such that for any
  $n$-point set $P \subset \R^d$, $\gamma_2(\AA_d(P)) \le C_d (1+\log
  n)^{d}$. Moreover, there exists a factorization of the incidence
  matrix of $\AA_d(P)$ achieving this bound into matrices $U$ and $V$
  with entries in the set $\{0,1\}$. 
\end{lemma}

To prove lower bounds on discrepancy, we use \eqref{eq:lb} and a dual
formula for $\gamma_2$:
\begin{equation}
  \label{eq:dual}
  \gamma_2(A) = \max\{\|PAQ\|_{\tr}: P, Q \text{ diagonal }, \tr P^2 =
  \tr Q^2 = 1\},
\end{equation}
where $\|M\|_{\tr}$ is the trace norm of the matrix $M$, equal to the
sum of singular values. For a proof of this formula,
see~\cite{LeeSS08,apx-disc}. In this paper, we use the easily proved special
case derived from \eqref{eq:dual} by setting $P = Q = \frac1n I$:
\begin{equation}
  \label{eq:dual-simple}
  \gamma_2(A) \ge \frac{1}{n} \|A\|_{\tr}.
\end{equation}
See~\cite{LinialMSS07-signmatrices} for a short direct proof of this
inequality. 

The $\gamma_2$ norm is also directly connected to upper and lower
bounds for data structures and differentially private mechanisms for
range searching and range counting. It follows immediately from the
definitions (see Section~\ref{sect:ds}) that the update time $t_u$
and the query time $t_q$ for an oblivious range searching data
structure with multiplicity $\Delta$ satisfy $\sqrt{t_ut_q} \ge
\Delta^{-1}\gamma_2(\FF(P))$; moreover, if $\gamma_2(\FF(P))$ is
achieved by a factorization into matrices with entries in $\{-\Delta,
\ldots, \Delta\}$, then $\sqrt{t_u t_q} \le \gamma_2(\FF(P))$. In
differential privacy, we have the following theorem (see
Section~\ref{sect:dp} for the definitions).
\begin{theorem}\label{thm:dp-gamma2}
  There exists an absolute constant such that the following holds for
  all small enough $\varepsilon$, and $\delta$ small enough with
  respect to $\varepsilon$. 
  Let $\FF$ be a collection of subsets of $\R^d$ and let $P$ be an
  $N$-point set in $\R^d$. Then, 
  \[
  \frac{1}{C\log |\FF(P)|} \frac{\gamma_2(\FF(P))}{\varepsilon} 
  \le \mathrm{opt}_{\varepsilon, \delta}(\FF(P))
  \le 
  C\sqrt{(\log |\FF(P)|) (\log 1/\delta)}\frac{\gamma_2(\FF(P))}{\varepsilon} 
  \]
\end{theorem}
Theorem~\ref{thm:dp-gamma2} was essentially proved, although stated
differently, in~\cite{NTZ}. For a statement equivalent to the one
above, and a proof, see Theorem 7.1.~in~\cite{thesis}. Note
that~\cite{thesis} uses the notation $\|\cdot \|_{E_\infty}$ in place
of the standard $\gamma_2(\cdot)$.

\subsection{Signed Series of Vectors}

We will use the following result of Banaszczyk:
\begin{lemma}[\cite{Bana13}]\label{lm:bana-steinitz}
  Let $v_1, \ldots, v_n \in \R^m$, $\forall i: \|v_i\|_2 \le 1$, and
  let $K \subset \R^m$ be a convex body symmetric around the
  origin. If $\mu_m(K) \ge 1 - 1/(2n)$, then there exists an assignment of signs
  $\chi\maps[n] \to \{-1, 1\}$ so that
  \[
  \forall j \in \{1, \ldots, n\}: \sum_{i = 1}^j{\chi(i) v_i} \in 5K.
  \]
\end{lemma}
This lemma was proved in the context of the well-known Steinitz
problem: given vectors $v_1, \ldots, v_n$, each of  Euclidean norm at
most $1$, such that $v_1 + \ldots + v_n = 0$, find a
permutation $\pi$ on $[n]$ such that for all integers $i$, $1 \le i
\le n$, $\|v_{\pi(1)} + \ldots + v_{\pi(i)}\|_2 \le C\sqrt{m}$, where
$C$ is an absolute constant independent of $m$ or
$n$. Lemma~\ref{lm:bana-steinitz} gives the best partial result in
this direction: it can be used to show a bound of $C(\sqrt{m} +
\sqrt{\log n})$ in place of $C\sqrt{m}$.

Lemma~\ref{lm:bana-steinitz} follows relatively easily from a
powerful result Banaszczyk proved in~\cite{banasz98}. Unfortunately,
the proof of the latter does not suggest any efficient algorithm to
find the signs $\chi(i)$, and no such algorithm is yet known, despite
some partial progress~\cite{BansalDG16}.

\subsection{Techniques}

Our approach builds on the connection between the $\gamma_2$ norm and
hereditary discrepancy shown in~\cite{apx-disc,disc-gamma2}. The new
idea which enables the tighter upper bound is the use of Banaszczyk's
signed series result (Lemma~\ref{lm:bana-steinitz}), in order to get
one dimension ``for free''. On a very high level, this is similar to
the way one dimension comes for free in constructions of point sets
with small geometric discrepancy, such as the Halton-Hammersley
construction mentioned before. The proof of
Theorem~\ref{thm:polytopes-ub} combines the ideas in the proof of
Theorem~\ref{thm:tusnady-ub} with a decomposition due to
Matou\v{s}ek~\cite{Matousek99}.

The lower bound for Tusn\'ady's problem in~\cite{disc-gamma2} relies
crucially on the product structure of $\AA_d$, and does not easily
extend to $\TT_B$ when $B$ is a polytope other than a box.  Instead,
for the lower bound in this paper, we combine the factorization norm
approach with the Fourier method, developed in discrepancy theory by
Beck: see~\cite{BeckChen} for an exposition. In order to give a lower
bound on $\gamma_2$, we use \eqref{eq:dual-simple}, and estimate the
trace norm of the incidence matrix $M$ of a set system related to
$\TT_B(P)$, where $P$ is a grid in $[0,1)^d$. We define $M$ so that it
is a convolution matrix, and its eigenvalues are given by the discrete
Fourier transform. While tight estimates are known for the average
decay of continuous Fourier coefficients of convex polytopes, we need
estimates on discrete Fourier coefficients, about which much less is
known. To bridge this gap, we prove a bound on the convergence rate of
discrete Fourier coefficients of convex polytopes to the continuous
Fourier coefficients. We also use an averaging argument in order to be
able to work with bounds on the average decay of the Fourier spectrum,
rather than having to estimate specific Fourier coefficients. These
techniques are general, and may be more widely applicable to geometric
combinatorial discrepancy questions. Our version of the Fourier method
has the curious feature that, even though we average over rotations of
the polytope $B$, in the end the lower bound holds for the set system
induced only by translations and dilations of $B$.

\section{Upper Bound for Tusn\'ady's Problem}


In this section we give the proof of Theorem~\ref{thm:tusnady-ub}.
  Let us fix the $n$-point set $P \subset \R^d$ once and for all. Without loss of
  generality, assume that each $p \in P$ has a distinct last
  coordinate, and order the points in $P$ in increasing order of their
  last coordinate as $p_1, \ldots, p_n$. Write each $p_i$ as $p_i =
  (q_i, r_i)$, where $q_i \in \R^{d-1}$ and $r_i \in \R$. With this
  notation, and the ordering we assumed, we have that $r_i < r_j$
  whenever $i <j$.

  Let $Q \eqdef \{q_i: 1\le i \le n\}$. Notice that this is an
  $n$-point set in $\R^{d-1}$. Denote the sets in
  $\AA_{d-1}(Q)$ as $A_1, \ldots, A_m$ (in no particular order). 
  By Lemma~\ref{lm:factor}, there exist vectors
  $u_1, \ldots, u_m$ and $v_1, \ldots, v_n$ such that
  \begin{equation}\label{eq:factor}
  \langle u_i, v_j\rangle = 
  \begin{cases}
    1 & q_j \in A_i\\
    0 & q_j \not \in A_i
  \end{cases},
  \end{equation}
  and $\|u_i\|_2 \le C_d(1+\log n)^{(d-1)}$, $\|v_j\|_2 \le 1$ for all
  $i$ and $j$. Define the
  symmetric polytope
  \[
  K \eqdef \{x \in \R^{m}: |\langle u_i, x \rangle| \le C'_d (1 + \log
  n)^{d-1/2}\ \forall i \in \{1, \ldots, m\}\},
  \]
  where $C'_d > C_d$ is a constant large enough that $\mu_{m}(K) \ge
  1 - 1/(2n)$. The fact that such a constant exists follows from
  standard concentration of measure results in Gaussian space. Indeed,
  using a Bernstein-type inequality for Gaussian measure, we can show
  that, for $C'_d$ big enough, $\mu_m(S_i) \ge 1 - 1/(2n^{d+1})$ for
  all $i\in [m]$, where $S_i \eqdef \{x: |\langle u_i, x \rangle| \le C'_d (1
  + \log n)^{d-1/2})\}$. By the union bound, since $m \le n^d$, this implies
  $\mu_m(K) = \mu_m(\bigcap_{i=1}^m S_i) \ge 1-1/2n$. 
  
  The body $K$ and the vectors $v_1, \ldots, v_n$ then satisfy the
  assumptions of Lemma~\ref{lm:bana-steinitz}, and, therefore, there
  exists an assignment of signs $\chi\maps [n] \to \{-1,1\}$ such that, for any
  $k$, $1 \le k \le n$,
  \[
  \sum_{j = 1}^k{\chi(j) v_j} \in 5K.
  \]
  By the definition of $K$,  this is equivalent to 
  \begin{equation}
    \label{eq:vect-balance}
    \forall i \in \{1, \ldots, m\}: 
    \left|\sum_{j = 1}^k{\chi(j) \langle u_i, v_j\rangle}\right| 
    \le 5C'_d (1 + \log  n)^{d-1/2}.
  \end{equation}

  For each $i$, $1\le i \le m$, let us define $A'_i = \{p_j: q_j \in
  A_i\}$. We claim that for any $x \in \R^d$, we can write $A(x) \cap
  P$ as $A'_i \cap \{p_1 , \ldots, p_k\}$ for some $i$ and $k$. (Here
  we assume that $A(x) \cap P$ is non-empty: the other case is
  irrelevant to the proof.) To see this, let $x = (y, x_d)$, where $y
  \in \R^{d-1}$ and $x_d \in \R$. Let $i$ be such that $A(y) \cap Q =
  A_i$, and let $k$ be the largest integer such that $r_k \le
  x_d$. Then:
  \[
  A(x) \cap P = \{p_j: q_j \in A(y), j \le k\} = A'_i \cap \{p_1,
  \ldots, p_k\}.
  \]
  It follows that 
  \begin{align*}
    \left|\sum_{j: p_j \in A(x)}{\chi(j)}\right| 
    &= \left|\sum_{j: q_j \in A_i, j \le  k}{\chi(j)}\right|
    =     \left|\sum_{j \le k}{\chi(j) \langle u_i, v_j\rangle}\right| 
    \le 5C'_d (1 + \log  n)^{d-1/2},
  \end{align*}
  where the penultimate equality follows from \eqref{eq:factor},  and
  the final inequality is \eqref{eq:vect-balance}. Since $x$ was arbitrary, we have shown that
  $\disc \AA_d(P) \le 5C'_d (1 + \log  n)^{d-1/2}$, as was required. 

\section{Upper Bound for an Arbitrary Polytope}

The main ingredient in extending Theorem~\ref{thm:tusnady-ub} to
arbitrary polytopes is a geometric decomposition due to
Matou\v{s}ek. To describe the decomposition we define the admissible
$k$-composition $\mathrm{AC}_k(\FF)$ of sets from a (finite) set system $\FF$ as
follows. For $k = 0$, $\mathrm{AC}_k(\FF) = \emptyset$; for an integer $k > 0$,
we have
\begin{align*}
  \mathrm{AC}_k(\FF) \eqdef &\{F_1 \cup F_2: F_1 \in \mathrm{AC}_{k_1}(\FF), F_2 \in  \mathrm{AC}_{k_2}(\FF), F_1 \cap F_2 = \emptyset, k_1 + k_2 = k\}\ \cup\\
  &\{F_1 \setminus F_2: F_1 \in \mathrm{AC}_{k_1}(\FF), F_2 \in  \mathrm{AC}_{k_2}(\FF), F_2 \subseteq F_1, k_1 + k_2 = k\}.
\end{align*}
By an easy induction on $k$, we see that 
\begin{align}
  \label{eq:combin}
  \disc \mathrm{AC}_k(\FF) &\le k\disc\FF,\\
  \label{eq:combin-gamma2}
  \gamma_2(\mathrm{AC}_k(\FF)) &\le k\gamma_2(\FF).
\end{align}

We also extend the notion of anchored boxes to 
``corners'' whose bounding hyperplanes are not necessarily
orthogonal. Let $W = \{w_1, \ldots, w_d\}$ be a basis of $\R^d$. Then
we define $\AA_W \eqdef \{A_W(x): x \in \R^d\}$, where $A_W(x) = \{y
\in \R^d: \langle w_i, y\rangle \le \langle w_i, x\rangle\ \forall i
\in [d]\}$. 

The following lemma gives the decomposition result we need.
\begin{lemma}[\cite{Matousek99}]\label{lm:decomp}
  Let $B \subset \R^d$ be a convex polytope. There exists a constant
  $k$ depending on $d$ and $B$, and $k$ bases $W_1, \ldots, W_k$ of
  $\R^d$ such that  every $B' \in \TT_B$ belongs to $\mathrm{AC}_k(\AA_{W_1} \cup
  \ldots \cup \AA_{W_k})$. Moreover, $e_1 \in W_1 \cap W_2 \cap \ldots
  \cap W_k$, where $e_1$ is the first standard basis vector of
  $\R^d$. 
\end{lemma}
Matou\v{s}ek does not state the condition after ``moreover'';
nevertheless, it is easy to verify this condition holds for the
recursive decomposition in his proof of Lemma~\ref{lm:decomp}.

We will also need a bound on $\gamma_2(\AA_W(P))$ for a basis $W$ and
an $n$-point set $P$. 

\begin{lemma}\label{lm:factor-W}
  For any $\ell \ge 1$ there exists a constant $C_\ell$ such that the
  following holds.  For any set $W$ of $\ell$ linearly independent
  vectors in $\R^d$, and any $n$-point set $P$, $\gamma_2(\AA_W(P))
  \le C_\ell (1+\log n)^{\ell}$. 
\end{lemma}
\begin{proof}
  Let $W = \{w_1, \ldots, w_\ell\}$, and let $u$ be the linear
  transformation from $\R^\ell$ to $\R^d$ that sends the $i$-th
  standard basis vector $e_i$ to $w_i$ for each $i \in [\ell]$. Let $Q
  = u^*(P) \eqdef \{u^*(p): p \in P\}$, where $u^*$ is the adjoint of
  $u$. It is easy to verify that $\AA_W(P) = \AA_\ell(Q)$, and,
  therefore,
  \[
  \gamma_2(\AA_W(P)) = \gamma_2(\AA_d(Q))  \le C_\ell (1+\log n)^{\ell},
  \]
  where the final inequality follows from Lemma~\ref{lm:factor}.
\end{proof}

As a warm-up, let us first prove an upper bound on
$\gamma_2(\TT_B(P))$. 

\begin{theorem}\label{thm:polytopes-ub-gamma2}
  For any $d \ge 1$ and any closed convex polytope $B \subset \R^d$
  there exists a constant $C_{d,B}$ such that for any set $P$ of $n$
  points in $\R^d$,
  \[
  \gamma_2(\TT_B(P)) \le C_{d,B}(1+\log n)^d.
  \]
\end{theorem}
\begin{proof}
  Let $W_1, \ldots, W_k$ be as in Lemma~\ref{lm:decomp}. By
  \eqref{eq:union} and Lemma~\ref{lm:factor-W}, 
  \[
  \gamma_2(\AA_{W_1}(P) \cup \ldots \cup \AA_{W_k}(P))
  \le\sqrt{k}C_d(1 + \log n)^{d}.
  \]
  By Lemma~\ref{lm:decomp}, $\TT_B(P) \subseteq
  \mathrm{AC}_k(\AA_{W_1}(P) \cup \ldots \cup \AA_{W_k}(P))$. Together
  with \eqref{eq:combin-gamma2} and the trivial fact that
  $\gamma_2(\FF') \le \gamma_2(\FF)$ whenever $\FF' \subseteq \FF$,
  this implies
  \begin{align*}
  \gamma_2(\TT_B(P)) &\le 
  \gamma_2(\mathrm{AC}_k(\AA_{W_1}(P) \cup \ldots \cup \AA_{W_k}(P)))\\
  &\le k \gamma_2(\AA_{W_1}(P) \cup \ldots \cup \AA_{W_k}(P))\\
  &\le k\sqrt{k}C_d(1 + \log n)^{d}.
  \end{align*}
  Since $k$ depends on $d$ and $B$ only, this finishes the proof of
  the theorem. 
\end{proof}

\begin{proof}[Proof of Theorem~\ref{thm:polytopes-ub}]
As in the proof of
Theorem~\ref{thm:tusnady-ub}, we fix the $n$-point set $P \subset
\R^d$ once and for all, and we order $P$ as $p_1, \ldots, p_n$ in
increasing order of the last coordinate. We write each $p_i$ as $p_i =
(q_i, r_i)$ for $q_i \in \R^{d-1}$ and $r_i \in \R$, and define $Q \eqdef
\{q_i: 1 \le i \le n\}$.

Let $W_1, \ldots, W_k$ be as in Lemma~\ref{lm:decomp}, and let $W'_i = W_i
\setminus \{e_1\}$ for each $i$, $1 \le i \le k$. Observe that $W'_i$
is a set of $d-1$ linearly independent vectors, and, by \eqref{eq:union} and
Lemma~\ref{lm:factor-W},
\[
\gamma_2(\AA_{W'_1}(Q) \cup \ldots \cup \AA_{W'_k}(Q))
 \le\sqrt{k}C_d(1 + \log n)^{d-1}.
\]
By an argument using Lemma~\ref{lm:bana-steinitz} analogous to the one
used in the proof of Theorem~\ref{thm:tusnady-ub}, we can then show that
there exists a constant $C'_{d,B}$ depending on $B$ and $d$ and a
coloring $\chi\maps [n] \to \{-1, 1\}$, such that for any integer $i$, $1 \le
i \le k$, and any $x \in \R^d$,
\[
\left| \sum_{j: p_j \in A_{W_i}(x)}{\chi(j)} \right| \le
C'_{d,B}(1+\log n)^{d-1/2}.
\]
Here $C'_{d,B}$ is implicitly assumed to depend on $k$ as well, which
depends on $B$ and $d$. This establishes that $\disc\FF \le
C'_{d,B}(1+\log n)^{d-1/2}$, where $\FF = \AA_{W_1}(P) \cup \ldots
\cup \AA_{W_k}(P)$. Because, by Lemma~\ref{lm:decomp}, $\TT_B(P) \subseteq
\mathrm{AC}_k(\FF)$, \eqref{eq:combin} implies
\[
\disc\TT_B(P) \le \disc \mathrm{AC}_k(\FF) \le k\disc\FF \le kC'_{d,B}(1+\log
n)^{d-1/2}. 
\]
This finishes the proof of the theorem.
\end{proof}

The same asymptotic bound with $\TT_B$ replaced by $\mathrm{POL}(\HH)$
for a family of hyperplanes $\HH$ can be proved by replacing
Lemma~\ref{lm:decomp} with an analogous decomposition lemma for
$\mathrm{POL}(\HH)$, also proved in~\cite{Matousek99}.

The upper bound on $\sqrt{t_ut_q}$ in Theorem~\ref{thm:ds} follows
from Theorem~\ref{thm:polytopes-ub-gamma2} and the observation that
the upper bound on $\gamma_2$ can be achieved by a factorization into
matrices with entries in $\{0,1\}$, which is equivalent to an
oblivious data structure with multiplicity $1$. The upper bound on
error in Theorem~\ref{thm:dp} follows from
Theorems~\ref{thm:dp-gamma2}~and~\ref{thm:polytopes-ub-gamma2}.

\section{Lower Bound}\label{sect:lb}

\junk{Let $T^d \eqdef (\R/\Z)^d$ be the
$d$-dimensional torus. We identify $T^d$ with the unit cube $[0, 1)^d$
in the natural way, i.e.~any coset in $(\R/\Z)^d$ is identified with
its unique representative in $[0, 1)^d$, and we endow $T^d$ with the
standard inner product and Euclidean norm inherited from $\R^d$. The
operations $x\pm y$ for $x,y\in T^d$ are interpreted as $(x\pm y)_i =
(x_i \pm y_i)\bmod 1$.}

In this section we prove lower bounds on $\disc(n,\TT_B)$ matching the
known lower bounds on $\disc(n,\AA_d)$ up to constants when $B$ is a
generic convex polytope (``generic'' is defined below). In order to
use Fourier analytic techniques, it will be convenient to work with a
modification of the incidence matrix of $\TT_B(P)$. To this end, let
us call a function $f\maps\R^d \to \R$ periodic if for every $x \in
[0,1)^d$ and every vector $y$ in the integer lattice $\Z^d$ we have
$f(x) = f(x+y)$. Let $Q_n \eqdef \{\frac{i}{n}\}_{i = 0}^{n-1}$. For a
periodic function $f\maps \R^d \to \R$, define a real $n^d\times n^d$
matrix $M(f,n)$ indexed by $Q_n^d$:
\[
m_{x,y}(f,n) \eqdef f(x-y),
\]
where, $x$ and $y$ range over $Q_n^d$.  \junk{$M(f,n)$ is the matrix
  (with respect to the standard basis) of the linear map given by
  convolution with the restriction $f|_{Q_n^d}$. } Of special interest
to us are periodic functions defined by convex polytopes $B \subseteq
[0,1)^d$. Let us define the function $f_B\maps\R^d \to R$ to be equal
to the indicator function $1_B$ of $B$ on $[0,1)^d$, extended
periodically to the rest of $\R^d$. For convenience, we use the
notation $M(B,n) \eqdef M(f_B, n)$.

There are two main observations that motivate studying
$M(B,n)$. First, each row of $M(B,n)$ is the indicator vector of the
disjoint union of most $2^d$ sets from $\TT_{-B}(Q_n^d) =
\TT_B(-Q_n^d) = \TT_B(Q_n^d)$, so any lower bound on the hereditary
discrepancy of $M(B,n)$ implies a lower bound on $\disc(n,\TT_B)$.
Second, because $M(B,n)$ is the matrix of a convolution operator, it
is diagonalized by the (discrete multidimensional) Fourier transform,
and we can use known results on the Fourier spectra of convex
polytopes to derive bounds on the trace norm of $M(B,n)$, and
therefore on $\gamma_2(M(B,n))$. 


Before we continue, let us introduce the standard notation for the
Fourier coefficients. For the
remainder of this section, we use $i = \sqrt{-1}$ to denote the
imaginary unit.
For any $u \in \R^d$ and a periodic function $f\maps \R^d \to \R$
integrable on $[0,1)^d$ we define the Fourier coefficient
\[
\hat{f}(\xi) \eqdef \int_{[0,1)^d}{f(x) e^{-2\pi i\langle \xi, x \rangle}dx}.
\]
We also define the discrete Fourier coefficients:
\[
\tilde{f}(\xi,n) \eqdef 
\frac{1}{n^d} \sum_{q \in Q_n}{f(x) e^{-2\pi i  \langle \xi, q\rangle}}.
\]
It is well known, and easy to verify, that the eigenvalues of $M(B,n)$
are given by $n^d \tilde{f}_B(\xi,n)$. There is quite a bit known
about the continuous Fourier coefficients $\hat{f}_B(\xi)$, but
comparatively less known about the discrete Fourier
coefficients. Intuitively, bounds on the Fourier coefficients are
easier to prove in the continuous domain because powerful tools like
the divergence theorem are available. In order to use the known bounds
on $\hat{f}_B(\xi)$, we estimate the rate of convergence of
$\tilde{f}_B(\xi,n)$ to $\hat{f}_B(\xi)$ with $n$. We follow an
approach similar to that used by Epstein~\cite{Epstein05} in the
one-dimensional setting. In order to adapt his results to our setting,
we need two additional ingredients: a higher-dimensional analog of
Jackson's theorem in approximation theory, and a continuous
approximation to the indicator function $f_B$. Next, we state the
higher-dimensional Jackson-type theorem (due to Yudin, also spelled
Judin) that we use.

\begin{definition}
  The \emph{modulus of continuity} of a  continuous function $f\maps
  \R^d \to \R$ is defined as
  $\omega(f,\delta) \eqdef \sup\{|f(x+t) - f(x)|: x,t\in \R^d,
  \|t\|_2\le \delta\}$.
\end{definition}

\begin{theorem}[\cite{Judin76}]\label{thm:apx}
  There exists a
  universal constant $C$ such that for any function $f$ and any
  integer $n \ge 1$ there exist an order $n$ trigonometric polynomial
  $p$ defined by
  \[
  p(x) = 
  \sum_{\nu \in \Z^d: \|\nu\|_\infty \le n}{c_\nu e^{2\pi i \langle \nu,x\rangle}}
  \]
  such that 
  \[
  \|f-p\|_\infty \eqdef 
  \sup_{x \in [0,1)^d}{|f(x) - p(x)|} \le  4\omega\left(f,\frac{C\sqrt{d}}{n}\right).
  \]
\end{theorem}

The next lemma is an analogue of Theorem~3.1.~in~\cite{Epstein05}. 

\begin{lemma}\label{lm:fourier-conv-cont}
  There exists a universal constant $C$ such that the following
  holds. Let $f\maps \R^d \to \R$ be a continuous periodic function
  with modulus of continuity $\omega(f,\delta)$. Then, for any $\xi \in \Z^d$ such that
  $\|\xi\|_\infty \le n$, we have
  \[
  |\tilde{f}(\xi, 2n+1) - \hat{f}(\xi)| \le
  8\omega\left(f,\frac{C\sqrt{d}}{n}\right).
  \]
\end{lemma}
\begin{proof}
  Let $p$ be the trigonometric polynomial of order $n$ guaranteed by
  Theorem~\ref{thm:apx}. By an elementary calculation, for any $\xi
  \in \Z^d$ such that $\|\xi\|_\infty \le n$, we have
  $\tilde{p}(\xi,2n+1) = \hat{p}(\xi) = c_\xi$, where $c_\xi$ is the
  coefficient of $e^{2\pi i \langle \xi, \cdot \rangle}$ in the
  expansion of $p$. For any $\xi$ as above, this gives us:
  \begin{align*}
  |\tilde{p}(\xi, 2n+1) - \hat{f}(\xi)| 
  &= |\hat{p}(\xi) - \hat{f}(\xi)|\\
  &= \left|\int_{[0,1)^d}{(p(x) - f(x)) e^{-2\pi i\langle \xi, x
        \rangle}dx}\right| \le \|f - p\|_\infty,
  \end{align*}
  where we used the trivial case of H\"older's inequality. Similarly,
  \[
  |\tilde{f}(\xi, 2n+1) - \tilde{p}(\xi,2n+1)| 
  = 
  \frac{1}{(2n+1)^d} \left| \sum_{q \in Q_n}{(f(x) - p(x)) e^{-2\pi i
        \langle \xi, q\rangle}}\right|
  \le \|f-p\|_\infty.
  \]
  Combining the two inequalities, and using the bound in
  Theorem~\ref{thm:apx}, we get
  \begin{align*}
  |\tilde{f}(\xi, 2n+1) - \hat{f}| &\le 
  |\tilde{f}(\xi, 2n+1) - \tilde{p}(\xi,2n+1)| +   |\tilde{p}(\xi,
  2n+1) - \hat{f}(\xi)|\\ 
  &\le 2\|f-p\|_\infty \le 8\omega\left(f,\frac{C\sqrt{d}}{n}\right).
  \end{align*}
  This completes the proof.
\end{proof}

In the next lemma we prove an explicit bound on how fast
$\tilde{f}_B(\xi,n)$ converges to $\hat{f}_B(\xi)$ with $n$ for a
convex polytope $B$. The bounds are most likely not tight, but
sufficient for our purposes, since we only need that the convergence
rate is polynomial in $\frac{1}{n}$. 

\begin{lemma}\label{lm:fourier-conv-polytope}
  Let $B \subseteq [0,1)^d$ be a convex polytope with non-empty
  interior. Then, for any $\xi \in \Z^d$, $\|\xi\|_\infty \le n$, we
  have
  \[
  |\tilde{f}_B(\xi, 2n+1) - \hat{f}_B(\xi)| =
  O_{d,B}\left(\frac{1}{\sqrt{n}}\right). 
  \]
\end{lemma}
\begin{proof}
  In order to apply Lemma~\ref{lm:fourier-conv-cont}, we need a
  continuous function $f$ which approximates $f_B$. We construct this
  approximation by, roughly, a piecewise linear interpolation along
  every direction.

  Let $c$ lie in the interior of $B$, and pick a real number $r > 0$
  such that $c+ rD^d \subseteq B$, where $D^d$ is the unit Euclidean
  ball in $\R^d$ centered at the origin. Let's define the gauge
  function $g$ by $g(x) = \inf\{t: x \in (1-t)c + tB\}$ on
  $[0,1)^d$. Note that $g(x) \le \frac{1}{r}\|x-c\|_2$, further, that
  for any $x,t\in \R^d$, $g(x+t) \le g(x) + g(t+c)$. Using these two
  observations, we get
  \[
  g(x+t) - g(x) \le g(t+c) \le \frac{1}{r}\|t\|_2.
  \]
  By symmetry, we also get that $g(x) - g(x+t) \le
   \frac{1}{r}\|t\|_2$. Therefore, $\omega(g, \delta) \le
   \frac{\delta}{r}$.  

   We define a periodic function $f$ on $[0,1)^d$ by
   \[
   f(x) =
   \begin{cases}
     1 & g(x) \le 1-\eps \\
     \frac{1}{\eps} (1-g(x)) & 1-\eps < g(x) <1 \\
     0 & g(x) \ge 1\\
   \end{cases},
   \]
   for a parameter $0 < \eps < 1$ which depends on $n$ and will be
   determined later. We then extend $f$ periodically to the rest of
   $\R^d$.
   The function $f$ is defined so that it is
   continuous and agrees with $f_B$ except for those $x$ for which
   $1-\eps < g(x) < 1$. Moreover, observe that
   \[
   \omega(f,\delta) \le \frac{1}{\eps} \omega(g,\delta) \le \frac{\delta}{r\eps}.
   \]
   Then, by Lemma~\ref{lm:fourier-conv-cont}, for any $\xi \in \Z^d$
   such that
   $\|\xi\|_\infty \le n$, we have
   \begin{equation}
   |\tilde{f}(\xi, 2n+1) - \hat{f}(\xi)| \le 
   \frac{8C\sqrt{d}}{r\eps n} \label{eq:apx-err}
   \end{equation}
   
   \begin{figure}[htp]
     \centering
     \includegraphics[scale=0.7]{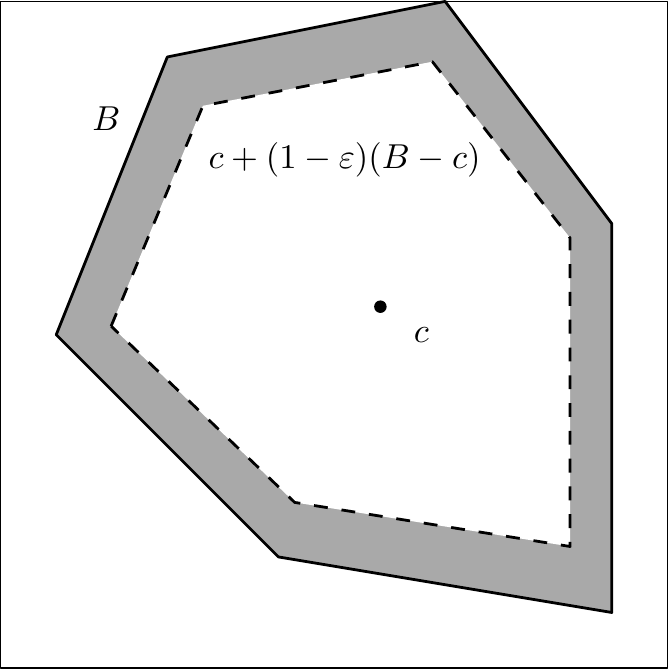}
     \caption{The set $S$ on which $f$ and $f_B$ disagree.}
     \label{fig:diff-set}
   \end{figure}

   It remains to bound $|\hat{f}_B(\xi) - \hat{f}(\xi)|$ and
   $|\tilde{f}_B(\xi, 2n+1) - \tilde{f}(\xi, 2n+1)|$. Let $S = \{x \in
   [0,1)^d: 1-\eps < g(x) < 1\}$ be the subset of $[0,1)^d$ on which
   $f$ and $f_B$ disagree. Notice that the closure of $S$ is $B
   \setminus (\eps c + (1-\eps)B) = c + (B-c) \setminus
   (1-\eps)(B-c)$ (see Figure~\ref{fig:diff-set}), so we have
   \[
   \lambda_d(S) = \lambda(B \setminus (1-\eps)B) 
   = (1 - (1-\eps)^d)\lambda_d(B) \le d\eps,
   \]
   where, in the final inequality, we used the assumption $B \subseteq
   [0,1)^d$, which implies $\lambda_d(B) \le 1$. 
   (Recall that we use $\lambda_d$ for the Lebesgue measure in
   $\R^d$.) We can now bound the first of our error terms using
   H\"older's inequality: for any $\xi \in \Z^d$
   \begin{align}
   |\hat{f}_B(\xi) - \hat{f}(\xi)| &= 
   \left|\int_{[0,1)^d}{(f_B(x) - f(x)) e^{-2\pi i\langle \xi, x
         \rangle}dx}\right|\notag\\
   &=   \left|\int_{S}{(f_B(x) - f(x)) e^{-2\pi i\langle \xi, x \rangle}dx}\right|
   \le \lambda_d(S) \le d \eps. \label{eq:cft-err}
   \end{align}
   A similar calculation for the discrete Fourier coefficients gives
   us
   \begin{align*}
   |\tilde{f}_B(\xi, 2n+1) - \tilde{f}(\xi, 2n+1)| &= 
   \frac{1}{(2n+1)^d} \sum_{q \in Q_n}{(f_B - f(x)) e^{-2\pi i  \langle   \xi, q\rangle}}\\
   &= \frac{1}{(2n+1)^d} \sum_{q \in Q_n \cap S}{(f_B - f(x)) e^{-2\pi i   \langle   \xi, q\rangle}}\\
   &\le \frac{|S \cap Q_n|}{(2n+1)^d}
   \end{align*}
   for any $\xi \in \Z^d$.
   By a standard volume argument, 
   \[
   |S \cap Q_n| \le 
   \frac{\lambda_d(S  + \frac{1}{2n}D^d)}{\lambda_d(\frac{1}{2n}D^d)}
   = (2n)^d \frac{\lambda_d(S  + \frac{1}{2n}D^d)}{\lambda_d(D^d)}.
   \]
   Because $\frac{1}{2n}D^d \subseteq \frac{1}{2rn}(B-c)$, we have $S +
   \frac{1}{2n}D^d \subseteq c + (1+\frac{1}{2rn})(B-c)\setminus(1-\eps
   - \frac{1}{2rn})(B-c)$, and, therefore,
   \[
   \lambda_d\left(S + \frac{1}{2n}D^d\right) 
   \le 
   \left(\left(1 + \frac{1}{2rn}\right)^d -  \left(1-\eps-\frac{1}{2rn}\right)^d\right)\lambda_d(B)
   \le 
   \eps d + \frac{3d}{2rn}.
   \]
   where the final inequality holds for $n \ge \frac{d}{2r}$.
   Putting the estimates together, we have
   \begin{equation}
   |\tilde{f}_B(\xi, 2n+1) - \tilde{f}(\xi, 2n+1)|
   \le 
   C_d \eps + \frac{C_d}{rn}, \label{eq:dft-err}
   \end{equation}
   for a constant $C_d$ depending on $d$ and all large enough $n$.

   Combining \eqref{eq:apx-err}, \eqref{eq:cft-err}, and
   \eqref{eq:dft-err}, for all large enough $n$ and any $\xi \in \Z^d$
   such that $\|\xi\|_\infty \le n$ we get
   \[
   |\tilde{f}_B(\xi, 2n+1) - \hat{f}_B(\xi, 2n+1)|
   \le \frac{8C\sqrt{d}}{r\eps n} + (C_d + d)\eps + \frac{C_d}{rn}.
   \]
   By setting $\eps = n^{-1/2}$, we get that the right hand side is in
   $O_{d,B}(n^{-1/2})$, as required. 
\end{proof}

Inspecting the proof of Lemma~\ref{lm:fourier-conv-polytope}, we see
that the only dependence on $B$ is via the radius $r$ of the Euclidean
ball contained in $B$: all other constants depend on the dimension
only. By John's theorem, we can apply an affine transformation to $B$
so that it is contained in $[0,1)^d$, and in fact in a Euclidean ball
of unit radius, and contains a ball of radius $\frac{1}{d}$. However,
the estimates we use below on the Fourier coefficients of $f_B$ do
depend on $B$, so we do not pursue this idea further.

We require a technical definition. 

\begin{definition}\label{defn:non-degenerate}
We will say that a polytope $B$ in $\R^d$ is \emph{generic} if
there exists a sequence of faces $F_1 \subseteq F_3 \subseteq \ldots
\subseteq F_{d-1}$ of $B$ such that $F_j$ is a face of dimension $j$,
for every $j \le d-2$ the face $F_j$ is not parallel to any other face
of $F_{j+1}$, and the facet $F_{d-1}$ is not parallel to any other
facet of $B$. 
\end{definition}

It is easy to see that $B$ is generic for
example if the facets of $B$ have normal vectors $u_1, \ldots, u_k$
that are in \emph{general position}, in the sense that for every $J
\subseteq [k]$ of size at most $d$ the set of vectors $\{u_j: j \in
J\}$ spans a subspace of dimension $|J|$. However, the condition of
being generic appears to be much weaker. We use the following
estimate on the Fourier coefficients of such a polytope.

\begin{theorem}[\cite{BrandoliniCT97}]\label{thm:fourier-estimate}
  Let $B$ be a generic polytope in $\R^d$ for $d \ge 2$. There exists
  a constant $c_{d,B}$, possibly depending on $d$ and $B$, such that
  for any $\rho \ge 1$
  \[
  \int_{S^{d-1}}{|\hat{f}_{B}(\rho\xi)| d\sigma_{d-1}(\xi)}
  \ge
  \frac{c_{d,B}\log^{d-1}(\rho)}{\rho^d} ,
  \]
  where $S^{d-1}$ is the unit Euclidean sphere in $\R^d$ and
  $\sigma_{d-1}$ is the uniform probability measure on $S^{d-1}$.
\end{theorem}

In~\cite{BrandoliniCT97} Theorem~\ref{thm:fourier-estimate} is stated
for a simplex, and after the proof the authors remark that the same
proof extends to any polytope which has one face not parallel to any
other face. However, in their proof, they induct on lower and lower
dimensional faces, and this condition needs to hold for any face they
induct on. Our definition of generic appears to be the minimal
condition under which their proof goes through. The authors
of~\cite{BrandoliniCT97} note that some genericity-like
condition is necessary since Theorem~\ref{thm:fourier-estimate} does
not hold for all values of $\rho$ when $B$ is a cube. Nevertheless, it
is likely that a variant of the theorem in which we average over
values of $\rho$ in a big enough interval could hold for arbitrary
polytopes, and may yield a lower bound with a constant that only
depends on $d$ and not on $B$.  We leave this extension for future
work.

Theorem~\ref{thm:fourier-estimate},
Lemma~\ref{lm:fourier-conv-polytope}, and an averaging argument yield
the following estimate on discrete Fourier coefficients, which is the
final step towards our lower bound.

\begin{lemma}\label{lm:nuclear}
  Let $B \subseteq c + D^d$ be a generic convex polytope with
  non-empty interior, where $D^d$ is the Euclidean unit ball in $\R^d$
  centered at the origin and $c$ is the centroid of $[0,1)^d$. There
  exists an orthogonal transformation $u$ such that $B_u \eqdef c + u(B-c)$
  satisfies
  \[
  \sum_{\xi \in \Z^d: \|\xi\|_\infty \le n}{|\tilde{f}_{B_u}(\xi,
    2n+1)|}
  = \Omega_{d,B}(\log^d(n)).
  \]
\end{lemma}
\begin{proof}
  Recall that we use $\mathbf{O}(d)$ to denote the group of orthogonal
  transformations on $\R^d$, and $\theta_d$ to denote the Haar
  probability measure on this group. We also use $\sigma_{d-1}$ for
  the uniform probability measure on the sphere. 

  For any $\xi \in \Z^d\setminus \{0\}$ we have
  \begin{align*}
  \int_{\mathbf{O}(d)}|\hat{f}_{B_u}(\xi)|d\theta(u)
  &=
  \int_{\mathbf{O}(d)}\left|\int_{c+u(B-c)}{e^{-2\pi i\langle \xi,
        x \rangle}dx}\right|d\theta(u)\\
  &=
  \int_{\mathbf{O}(d)}\left|\int_{u(B)}{e^{-2\pi i\langle \xi,
        x \rangle}dx}\right|d\theta(u)\\
  &=
  \int_{\mathbf{O}(d)}\left|\int_{B}{e^{-2\pi i\langle \xi,
        u^*(x) \rangle}dx}\right|d\theta(u)\\
  &=
  \int_{\mathbf{O}(d)}\left|\int_{B}{e^{-2\pi i\langle u(\xi),
        x \rangle}dx}\right|d\theta(u)\\
  &=
  \int_{S^{d-1}}\left|\int_{B}{e^{-2\pi i\langle \|\xi\|_2\zeta,
        x \rangle}dx}\right|d\sigma_{d-1}(\zeta)\\
  &\ge
  \frac{c\log^{d-1}(\|\xi\|_2)}{\|\xi\|_2^d}.
  \end{align*}
  Above, in the penultimate line we used the fact that for any
  measurable set $Y \subseteq S^{d-1}$, and any $y \in S^{d-1}$,
  $\sigma_{d-1}(Y) = \theta_d(\{u \in \mathbf{O}(d): u(y) \in
  Y\})$. The final inequality is implied by
  Theorem~\ref{thm:fourier-estimate} for an appropriate constant $c$
  depending on $d$ and $B$. Therefore,
  \[
  \int_{\mathbf{O}(d)}\left(
    \sum_{\xi \in \Z^d: 0 < \|\xi\|_2 = j}{|\hat{f}_{B_u}(\xi)|}\right)d\theta_d(u) 
  \ge c {m_j \frac{\log^{d-1}(j)}{j^{d/2}}},
  \]
  where $m_j = \{\xi \in \Z^d: \|\xi\|_2^2 = j\}$. By
  Lemma~\ref{lm:fourier-conv-polytope}, for all sufficiently large $n$
  and for all $j$  such that $j \le (c_1n)^{1/d}$ for a
  sufficiently small constant $c_1$, we have $\tilde{f}_{B_u}(\xi,
  2n+1) \ge \hat{f}_{B_u}(\xi) - \frac{c}{2j^{d/2}}$ for any $\xi \in
  \N_0^d$ for which $\|\xi\|_2^2 = j$. Therefore, 
  \[
  \int_{\mathbf{O}(d)}\left(
    \sum_{\xi \in \Z^d: 0<\|\xi\|_2^2 \le (c_1n)^{1/d}}
    {|\tilde{f}_{B_u}(\xi, 2n+1)|}\right)d\theta_d(u) 
  \ge c \sum_{j = 1}^{\lfloor (c_1n)^{1/d} \rfloor} {m_j
    \frac{\log^{d-1}(j)}{2j^{d/2}}}. 
  \]
  There must then exist a choice of $u \in \mathbf{O}(d)$ such that
  \begin{equation}\label{eq:fourier-cont-lb}
    \sum_{\xi \in \N_0^d: \|\xi\|_2^2 \le (c_1n)^{1/d}}
    {|\tilde{f}_{B_u}(\xi, 2n+1)|}
    \ge c \sum_{j = 1}^{\lfloor (c_1n)^{1/d} \rfloor} 
    {m_j \frac{\log^{d-1}(j)}{2j^{d/2}}}. 
  \end{equation}
  Let us fix such a $u$ for the rest of the proof. We proceed to
  estimate the right hand side of~\eqref{eq:fourier-cont-lb}. Define
  $\ell = \lfloor (c_1n)^{1/d} \rfloor$ to be the upper bound of the
  summation, and let $k = \lfloor \sqrt{\ell} \rfloor$. Let $m_{\le j}
  = m_1 + \ldots + m_j$. Using summation by parts, we have
  \begin{align*}
  \sum_{j = 1}^{\ell}{m_j   \frac{\log^{d-1}(j)}{j^{d/2}}}
  &\ge
  \log^{d-1}(k)\sum_{j = k}^{\ell}{\frac{m_j}{j^{d/2}}}\\
  &= 
  \log^{d-1}(k)  \frac{m_{\le \ell}}{\ell^{d/2}}
  +
  \log^{d-1}(k)
  \sum_{j = k}^{\ell-1}{m_{\le j}\left(\frac{1}{j^{d/2}} - \frac{1}{(j+1)^{d/2}}\right)}.
  \end{align*}
  By standard estimates (e.g.~Minkowski's first theorem), there exists
  a constant $c_2$ depending on the dimension $d$ such that $m_{\le j}
  \ge c_2 j^{d/2}$. By convexity, $\frac{1}{j^{d/2}} -
  \frac{1}{(j+1)^{d/2}} \ge \frac{d/2}{(j+1)^{(d+2)/2}}$. Plugging
  these inequalities into the bound above, we get that 
  \[
  \sum_{j =  1}^{\ell}{m_j   \frac{\log^{d-1}(j)}{j^{d/2}}} 
  \ge 
  c_2 \log^{d-1}(k) \sum_{j = k}^\ell{\frac{j^{d/2}}{(j+1)^{(d+2)/2}}}
  = \Omega_d(\log^d n).
  \]
 Together with \eqref{eq:fourier-cont-lb}, this completes the
  proof of the lemma. 
\end{proof}

\begin{proof}[Proof of Theorem~\ref{thm:lb-main}]
  By scaling we can assume that $B \subseteq c + D^d$, where 
  $D^d$ is the Euclidean unit ball in $\R^d$ centered at the origin
  and $c$ is the centroid of $[0,1)^d$. Let then $u$ be the orthogonal
  transformation given by Lemma~\ref{lm:nuclear}, and let $M \eqdef
  M(B_u, 2n+1)$. It is easy to verify that $M$ is diagonalized by the
  collection of orthogonal eigenvectors $\{(e^{2\pi i\langle \xi,
    x\rangle})_{x \in Q_{2n+1}^d}: \xi \in \Z^d, \|\xi\|_\infty \le
  n\}$, and the eigenvalue associated with the eigenvector $(e^{2\pi
    i\langle \xi, x\rangle})_{x \in Q_{2n+1}^d}$ is $(2n+1)^d
  \tilde{f}_{B_u}(\xi, 2n+1)$. Since $M$ is a normal matrix,
  i.e.~$M^\intercal M = MM^\intercal$, or, equivalently, since it is
  diagonalized by a system of orthogonal eigenvectors, its singular
  values are equal to the absolute values of its eigenvalues. By
  Lemma~\ref{lm:nuclear} we have the estimate
  \begin{align}
    \gamma_2(M) &\ge \frac{1}{(2n+1)^d} \|M\|_{\tr}\notag\\
    &=   \sum_{\xi \in \Z^d: \|\xi\|_\infty \le n}{|\tilde{f}_{B_u}(\xi,
    2n+1)|}
  = \Omega_{d,B}(\log^d(n)).    \label{eq:gamma2-M-lb}
  \end{align}
  Let $A$ be the incidence matrix of $\TT_{B_u}(-Q_{2n+1})$. Notice
  that every row of $M$ can be represented as the disjoint sum of at
  most $2^d$ rows of $A$. Therefore, we can write $M = A_1 + \ldots +
  A_k$, where $k \le 2^d$, and each row of each matrix $A_j$ is also a
  row of $A$. Since duplication and rearrangement of rows preserve
  $\gamma_2$, and dropping rows does not increase it, by the triangle
  inequality for $\gamma_2$ we have
  \[
  \gamma_2(M) \le \sum_{j=1}^k \gamma_2(A_j) 
  \le 2^d\gamma_2(A) = 2^d \gamma_2(\TT_{B_u}(-Q_{2n+1})).
  \]
  Define the pointset $P = \{u^*(x-c) + c: x \in -Q_{2n+1}\}$, and
  notice that $\TT_B(P) = \TT_{B_u}(-Q_{2n+1})$, so, by the inequality
  above and \eqref{eq:gamma2-M-lb}, we have 
  \begin{equation}\label{eq:gamma2-lb}
  \gamma_2(\TT_B(P)) = \Omega_{d,B}(\log^d n). 
  \end{equation}
  Equations~\eqref{eq:lb}~and~\eqref{eq:gamma2-lb} imply
  \[
  \disc(n^d, \TT_B) \ge \herdisc \TT_B(P)
  = \Omega_{d,B}(\log^{d-1}n).
  \]
  Therefore, $\disc(n, \TT_B) = \Omega_{d,B}(\log^{d-1}n)$, as was to
  be proved.
\end{proof}

Equation~\eqref{eq:gamma2-lb} implies the lower bound on
$\sqrt{t_ut_q}$ in Theorem~\ref{thm:ds}. The lower bound on error in
Theorem~\ref{thm:dp} follows from equation~\eqref{eq:gamma2-lb} and
Theorem~\ref{thm:dp-gamma2}. 

\section*{Acknowledgements}

This research was partially supported by NSERC. The author would like to thank the organizers of the 2016 discrepancy theory workshop in Varenna, where some of the initial ideas in this paper were conceived. The author would also like to thank Jozsef Beck for suggesting the problem of extending the lower bound for Tusn\'ady's problem to arbitrary polytopes, and Kunal Talwar for some initial discussions of this problem. 

\bibliographystyle{alpha}
\bibliography{tusnady}

\begin{thebibliography}{DMNS06}

\bibitem[ABN17]{mcqmc}
Christoph Aistleitner, Dmitriy Bilyk, and Aleksandar Nikolov.
\newblock Tusn{\'{a}}dy's problem, the transference principle, and non-uniform
  {QMC} sampling.
\newblock {\em CoRR}, abs/1703.06127, 2017.

\bibitem[Ale91]{Alexander91}
Ralph Alexander.
\newblock Principles of a new method in the study of irregularities of
  distribution.
\newblock {\em Invent. Math.}, 103(2):279--296, 1991.

\bibitem[Ban98]{banasz98}
W.~Banaszczyk.
\newblock Balancing vectors and {G}aussian measures of $n$-dimensional convex
  bodies.
\newblock {\em Random Structures and Algorithms}, 12(4):351--360, 1998.

\bibitem[Ban12]{Bana13}
Wojciech Banaszczyk.
\newblock On series of signed vectors and their rearrangements.
\newblock {\em Random Structures Algorithms}, 40(3):301--316, 2012.

\bibitem[BC08]{BeckChen}
J\'ozsef Beck and William W.~L. Chen.
\newblock {\em Irregularities of distribution}, volume~89 of {\em Cambridge
  Tracts in Mathematics}.
\newblock Cambridge University Press, Cambridge, 2008.
\newblock Reprint of the 1987 original [MR0903025].

\bibitem[BCT97]{BrandoliniCT97}
Luca Brandolini, Leonardo Colzani, and Giancarlo Travaglini.
\newblock Average decay of {F}ourier transforms and integer points in
  polyhedra.
\newblock {\em Ark. Mat.}, 35(2):253--275, 1997.

\bibitem[BDG16]{BansalDG16}
Nikhil Bansal, Daniel Dadush, and Shashwat Garg.
\newblock An algorithm for koml{\'{o}}s conjecture matching banaszczyk's bound.
\newblock In Irit Dinur, editor, {\em {IEEE} 57th Annual Symposium on
  Foundations of Computer Science, {FOCS} 2016, 9-11 October 2016, Hyatt
  Regency, New Brunswick, New Jersey, {USA}}, pages 788--799. {IEEE} Computer
  Society, 2016.

\bibitem[Bec81]{Beck81}
J.~Beck.
\newblock Balanced two-colorings of finite sets in the square. {I.}
\newblock {\em Combinatorica}, 1:327--335, 1981.

\bibitem[BG16]{BansalG16}
Nikhil Bansal and Shashwat Garg.
\newblock Algorithmic discrepancy beyond partial coloring.
\newblock {\em CoRR}, abs/1611.01805, 2016.
\newblock To appear at STOC 2016.

\bibitem[BLV08]{BilykLV08}
Dmitriy Bilyk, Michael~T. Lacey, and Armen Vagharshakyan.
\newblock On the small ball inequality in all dimensions.
\newblock {\em J. Funct. Anal.}, 254(9):2470--2502, 2008.

\bibitem[Cha00]{Chazelle-book}
Bernard Chazelle.
\newblock {\em The discrepancy method}.
\newblock Cambridge University Press, Cambridge, 2000.
\newblock Randomness and complexity.

\bibitem[CNN11]{dischard}
Moses Charikar, Alantha Newman, and Aleksandar Nikolov.
\newblock Tight hardness results for minimizing discrepancy.
\newblock In Dana Randall, editor, {\em Proceedings of the Twenty-Second Annual
  {ACM-SIAM} Symposium on Discrete Algorithms, {SODA} 2011, San Francisco,
  California, USA, January 23-25, 2011}, pages 1607--1614. {SIAM}, 2011.

\bibitem[DMNS06]{DMNS}
C.~Dwork, F.~Mcsherry, K.~Nissim, and A.~Smith.
\newblock Calibrating noise to sensitivity in private data analysis.
\newblock In {\em TCC}, 2006.

\bibitem[DR14]{DworkR14}
Cynthia Dwork and Aaron Roth.
\newblock The algorithmic foundations of differential privacy.
\newblock {\em Foundations and Trends in Theoretical Computer Science},
  9(3-4):211--407, 2014.

\bibitem[Drm96]{Drmota96}
Michael Drmota.
\newblock Irregularities of distributions with respect to polytopes.
\newblock {\em Mathematika}, 43(1):108--119, 1996.

\bibitem[Eps05]{Epstein05}
Charles~L. Epstein.
\newblock How well does the finite {F}ourier transform approximate the
  {F}ourier transform?
\newblock {\em Comm. Pure Appl. Math.}, 58(10):1421--1435, 2005.

\bibitem[Fre82]{Fredman82}
Michael~L. Fredman.
\newblock The complexity of maintaining an array and computing its partial
  sums.
\newblock {\em J. {ACM}}, 29(1):250--260, 1982.

\bibitem[G\"02]{Gotz02}
M.~G\"otz.
\newblock Discrepancy and the error in integration.
\newblock {\em Monatsh. Math.}, 136(2):99--121, 2002.

\bibitem[Hal60]{Halton60}
J.~H. Halton.
\newblock On the efficiency of certain quasi-random sequences of points in
  evaluating multi-dimensional integrals.
\newblock {\em Numer. Math.}, 2:84--90, 1960.

\bibitem[Ham60]{Hammersley60}
J.~M. Hammersley.
\newblock Monte {C}arlo methods for solving multivariable problems.
\newblock {\em Ann. New York Acad. Sci.}, 86:844--874 (1960), 1960.

\bibitem[Jud76]{Judin76}
V.~A. Judin.
\newblock A multidimensional {J}ackson theorem.
\newblock {\em Mat. Zametki}, 20(3):439--444, 1976.
\newblock Presented at the International Conference on the Theory of
  Approximation of Functions, held at Kaluga, July 24--28, 1975.

\bibitem[Lar14]{larsen14}
K.~G. Larsen.
\newblock On range searching in the group model and combinatorial discrepancy.
\newblock {\em SIAM Journal on Computing}, 43(2):673--686, 2014.

\bibitem[LMSS07]{LinialMSS07-signmatrices}
Nati Linial, Shahar Mendelson, Gideon Schechtman, and Adi Shraibman.
\newblock Complexity measures of sign matrices.
\newblock {\em Combinatorica}, 27(4):439--463, 2007.

\bibitem[LS{\v{S}}08]{LeeSS08}
Troy Lee, Adi Shraibman, and Robert {\v{S}}palek.
\newblock A direct product theorem for discrepancy.
\newblock In {\em Proceedings of the 23rd Annual {IEEE} Conference on
  Computational Complexity, {CCC} 2008, 23-26 June 2008, College Park,
  Maryland, {USA}}, pages 71--80. {IEEE} Computer Society, 2008.

\bibitem[Mat99]{Matousek99}
Jir{\'{\i}} Matou\v{s}ek.
\newblock On the discrepancy for boxes and polytopes.
\newblock {\em Monatsh. Math.}, 127(4):325--336, 1999.

\bibitem[Mat10]{Matousek-book}
Ji\v{r}\'{\i} Matou\v{s}ek.
\newblock {\em Geometric discrepancy}, volume~18 of {\em Algorithms and
  Combinatorics}.
\newblock Springer-Verlag, Berlin, 2010.
\newblock An illustrated guide, Revised paperback reprint of the 1999 original.

\bibitem[MN15]{e8-tusnady}
Jir{\'{\i}} Matousek and Aleksandar Nikolov.
\newblock Combinatorial discrepancy for boxes via the gamma{\_}2 norm.
\newblock In Lars Arge and J{\'{a}}nos Pach, editors, {\em 31st International
  Symposium on Computational Geometry, SoCG 2015, June 22-25, 2015, Eindhoven,
  The Netherlands}, volume~34 of {\em LIPIcs}, pages 1--15. Schloss Dagstuhl -
  Leibniz-Zentrum fuer Informatik, 2015.

\bibitem[MNT15]{disc-gamma2}
Ji\v{r}\'{\i} Matou\v{s}ek, Aleksandar Nikolov, and Kunal Talwar.
\newblock Factorization norms and hereditary discrepancy.
\newblock {\em CoRR}, abs/1408.1376v2, 2015.

\bibitem[Nik14]{thesis}
Aleksandar Nikolov.
\newblock {\em New Computational Aspects of Discrepancy Theory}.
\newblock PhD thesis, Rutgers, The State University of New Jersey, 2014.

\bibitem[NT15]{apx-disc}
Aleksandar Nikolov and Kunal Talwar.
\newblock Approximating hereditary discrepancy via small width ellipsoids.
\newblock In Piotr Indyk, editor, {\em Proceedings of the Twenty-Sixth Annual
  {ACM-SIAM} Symposium on Discrete Algorithms, {SODA} 2015, San Diego, CA, USA,
  January 4-6, 2015}, pages 324--336. {SIAM}, 2015.

\bibitem[NTZ16]{NTZ}
Aleksandar Nikolov, Kunal Talwar, and Li~Zhang.
\newblock The geometry of differential privacy: The small database and
  approximate cases.
\newblock {\em {SIAM} J. Comput.}, 45(2):575--616, 2016.

\bibitem[Rot54]{Roth54}
K.~F. Roth.
\newblock On irregularities of distribution.
\newblock {\em Mathematika}, 1:73--79, 1954.

\bibitem[Sch72]{Schmidt-VII}
Wolfgang~M. Schmidt.
\newblock Irregularities of distribution. {VII}.
\newblock {\em Acta Arith.}, 21:45--50, 1972.

\bibitem[TJ89]{TJ-book}
N.~Tomczak-Jaegermann.
\newblock {\em Banach-Mazur Distances and Finite-Dimensional Operator Ideals}.
\newblock Pitman Monographs and Surveys in Pure and Applied Mathematics 38. J.
  Wiley, New York, 1989.

\end{thebibliography}
\end{document}